\documentclass[11pt]{amsart}
\usepackage{amsmath}
\usepackage{amsthm}
\usepackage{amssymb}
\usepackage{amsfonts,mathrsfs}
\usepackage{stmaryrd}
\usepackage{amsxtra}  
\usepackage{epsfig}
\usepackage{verbatim}
\usepackage[all]{xy}
\usepackage{hyperref}
\hypersetup{colorlinks=true, linkcolor = blue, citecolor = red}

\usepackage{tikz}

\theoremstyle{plain}
\newtheorem{theorem}[equation]{Theorem}
\newtheorem{proposition}[equation]{Proposition}
\newtheorem{lemma}[equation]{Lemma}
\newtheorem{corollary}[equation]{Corollary}
\newtheorem{definition}[equation]{Definition}
\theoremstyle{remark}
\newtheorem{remark}[equation]{Remark}
\numberwithin{equation}{section}

\newcommand{\dbar}{\bar \partial}

\newcommand{\cb}{{\mathcal B}}

\newcommand{\sd}{{\mathscr D}}

\newcommand{\ve}{\varepsilon}

\newcommand{\C}{{\mathbb C}}
\newcommand{\D}{{\mathbb D}}

\newcommand{\h}{{\mathbb H}}

\newcommand{\R}{{\mathbb R}}

\newcommand{\Z}{{\mathbb Z}}

\begin{document}

\title[$\dbar$ on Hartogs]{A solution operator for $\dbar$ on the Hartogs triangle and $L^p$ estimates}
\author{L. Chen \& J. D. McNeal}
\begin{abstract} An integral solution operator for $\dbar$ is constructed on product domains that include the punctured bidisc. This operator is shown to satisfy $L^p$ estimates for all $1\leq p <\infty$, though with non-standard -- relative to strongly pseudoconvex
domains -- bounding term.
These estimates imply $L^p$ estimates for $\dbar$ on the Hartogs triangle, with greater range of $p$ than the canonical solution satisfies.
 \end{abstract}
\address{Department of Mathematics, \newline The Ohio State University, Columbus, Ohio, USA}
\email{chen.1690@osu.edu}
\address{Department of Mathematics, \newline The Ohio State University, Columbus, Ohio, USA}
\email{mcneal@math.ohio-state.edu}

\maketitle 


\section*{Introduction}\label{S:intro}

Let $\Omega$ be a domain in $\C^n$. If $\alpha$ is a $\dbar$-closed $(0,1)$-form on $\Omega$, consider
solutions to the Cauchy-Riemann system $\dbar v=\alpha$. A fundamental problem is to determine whether particular solutions
satisfying norms estimates exist for various norms. Such solutions lead to construction of non-trivial holomorphic functions
on $\Omega$.

Solving $\dbar$ with estimates depends both on the geometry of $\Omega$ and the norms considered. In this paper
results on two classes of domains are established -- product domains, especially with non-smooth factors, and the Hartogs triangle -- with estimates in $L^p$ norms,
$1 \le p<\infty$. Obtaining $L^p$ estimates for $\dbar$ on the Hartogs triangle motivated our investigation and is achieved in Theorem \ref{HarLp'}. However this result is
proved by transferring the $\dbar$ problem to a 2-dimensional product domain, so $L^p$ estimates for $\dbar$ on product spaces are established first. The results on product domains are 
new and of independent interest. That such estimates were not previously established is unusual, given the success of integral formulas on domains with more complicated
geometry. The study of $\dbar$ on product spaces accounts for most of the paper's length.

A successful method for obtaining non-$L^2$ estimates on $\dbar$ starts by establishing integral representation formulas with holomorphic kernels
for forms.\footnote{Or almost holomorphic kernels; see \cite{Range13} for a result in this direction.}
This approach was inaugurated by \cite{Henkin69} and \cite{GrauertLieb}
on strongly pseudo convex domains and was intensely pursued in the two decades after  \cite{Henkin69}, \cite{GrauertLieb}; see \cite{Kerzman71}, \cite{Lieb70}, and \cite{Ovrelid}
for some early foundational results. There are many significant results in this direction, too numerous to survey; see \cite{Range86} and \cite{HenkinLeiterer} for references to
the main results prior to 1985.

Integral formulas follow from a general procedure, the Cauchy-Fantappi\' e method, once a {\it generating form} is constructed; see \cite{Range86}, \cite{RangeSiu}, and \cite{LanSte13}.
However Cauchy-Fantappi\' e integral formulas have almost exclusively been derived for domains with smooth boundary (plus additional, restrictive geometric conditions) because Stokes' theorem is freely applied during the construction.
Two notable exceptions are \cite{RangeSiu}, on strongly pseudoconvex domains with piecewise smooth boundary, and \cite{Henkin71}, on analytic polyhedra.

For a product domain, the boundary is not smooth nor strongly pseudoconvex away from its boundary singularities. So while many techniques used below are well-known, modifications of the ``standard recipe'' are also required to establish our integral formulas. The first goal is to obtain the abstract integral formula \eqref{T1} on a 2-dimensional product domain with smoothly bounded factors; the derivation crucially uses an
idea from  \cite{RangeSiu}. A formula for products with higher-dimensional factors is also obtained, see Remark \ref{R:productSolution}.
The 2-dimensional formula is then converted into an explicit solution operator using the Cauchy generating form, when the data is sufficiently smooth:

\begin{proposition}\label{I:main} Suppose $D_1, D_2\subset\C$ are domains with $C^1$ boundary. If $f\in C^1_{0,1}\left(\overline{D_1\times D_2}\right)$ satisfies $\dbar f=0$, the function
\begin{equation}
\label{I:derivativeT1}
\begin{split}
T(f)
&=\frac{-1}{2 \pi i}\int_{D_2}\frac{f_2(z_1,\zeta_2)}{\zeta_2-z_2}d\bar \zeta_2\wedge d\zeta_2+\frac{-1}{2 \pi i}\int_{D_1}\frac{f_1(\zeta_1,z_2)}{\zeta_1-z_1}d\bar \zeta_1\wedge d\zeta_1\\
&+\frac{-1}{(2 \pi i)^2}\int_{D_1 \times D_2}\frac{\sd(f)(\zeta_1,\zeta_2)}{(\zeta_1-z_1)(\zeta_2-z_2)}\,d\bar \zeta_1\wedge d\zeta_1\wedge d\bar \zeta_2\wedge d\zeta_2
\end{split}
\end{equation}
solves $\dbar (Tf) =f$. In \eqref{I:derivativeT1}, $\sd f =\frac{\partial f_1}{\partial \bar z_2}=\frac{\partial f_2}{\partial \bar z_1}$. 
\end{proposition}

This is proved as Proposition \ref{P:strong} below. Even in the case of the bidisc, i.e., when $D_j=\D=\left\{z\in\C: |z|<1\right\}$, 
Proposition \ref{I:main} is new. Since the kernels in the integrands are Cauchy kernels, or iterated Cauchy kernels, mapping properties of the operator $T$ are easy to derive.
Additionally, a re-expression of \eqref{I:derivativeT1} -- see Remark \ref{T1forDxA} -- corrects a minor error in a formula displayed on page 212 of \cite{Henkin71},  given without proof, and reproduced in \cite{FornaessLeeZhang}. (The error is inconsequential for the estimates proved in  \cite{FornaessLeeZhang}.)

 The solution operator \eqref{I:derivativeT1} is extended to non-$C^1$ bounded domains -- including $\D\times\D^*$, where $\D^*=\left\{ 0 <|z| <1\right\}$ is the punctured disc -- and to forms not necessarily smooth
 up to the boundary in Section \ref{SS:punctured}. The $L^p$ boundedness of the integral operators is also proved in Section \ref{LpofT}. In contrast to results on strongly pseudoconvex domains, non-standard $L^p$ boundedness of the data is needed to obtain an $L^p$ bound on the solution. Define the norm $\|f\|_{\cb}:=\|f_1\|_{L^p(D_1\times D_2)}+\|f_2\|_{L^p(D_1\times D_2)}+\|\sd(f)\|_{L^p(D_1\times D_2)}$ on $(0,1)$-forms $f=f_1d\bar z_1+ f_2 d\bar z_2$, see Definition \ref{D:banach}.
 The main $L^p$ result on product domains, Theorem \ref{Lpwithdbarf}, says the ordinary $L^p$ norm of $Tf$ is dominated by  $\|f\|_{\cb}$. 
 
 In Section \ref{S:necessity}, the condition $f\in L^p\left(D_1\times D_2\right)$ alone is shown not to be sufficient to conclude $Tf\in L^p$. More dramatically, the example there shows that $Tf$ can fail to {\it exist} for
 $f\in L^1\left(D_1\times D_2\right)$, or more generally for $f\in L^p$, $1\leq p<2$. This contrasts sharply with results on the Henkin-Cauchy-Fantappi\' e operator known on strongly pseudoconvex domains \cite{Kerzman71}, \cite{Ovrelid},
 some finite type domains \cite{ChaNagSte}, \cite{FefKohn88}, and even some infinite type domains \cite{HaKhaRai}. The contrast is interesting and should be understood more fully.  The observation in Section \ref{S:necessity}
 merely inaugurates this new phenomena; finding actual necessary conditions on $f$ that follow from $Tf$ existing or satisfying $L^p$ estimates remains open. Such conditions, beyond $f\in L^p\left(D_1\times D_2\right)$, obviously have consequences when using the estimates in application. To be clear: our computations in Section \ref{S:necessity} are made only on the Henkin solution operator.
 
The main previous result on $\dbar$-estimates for the bidisc are the $L^\infty$ estimates in  \cite{FornaessLeeZhang}. There is a point connecting  \cite{FornaessLeeZhang}, the undetermined necessary conditions mentioned above, and the older literature on the Henkin solution. Norm control of derivatives of the data $f$ is assumed in \cite{FornaessLeeZhang}, though somewhat obliquely. In that paper the estimate 
$\|Tf\|_{L^\infty(\D\times\D)} \leq\, C\|f\|_{L^\infty(\D\times\D)}$ is proved, but only under the assumption that $f\in C^1_{0,1}\left(\overline{\D\times\D}\right)$. Thus the question posed by Kerzman in 1971 -- does there exist a solution operator satisfying $\|Sf\|_{L^\infty}\leq C \|f\|_{L^\infty}$ for {\it all} $\dbar$-closed forms in $L^\infty$, \cite{Kerzman71} remark on pages 311--312 -- is still unresolved on the bidisc.

 Finally in Section \ref{S:Hartogs}, the 
 biholomorphism between the Hartogs triangle $\h$ and $\D\times\D^*$ transfers the $\dbar$ problem on $\h$ to a $\dbar$ problem on $\D\times\D^*$ with different data. $L^p$ estimates for a solution operator on $\h$ are then inferred from those on $\D\times\D^*$, cf. Theorem \ref{HarLp'}. There are previous results about solving $\dbar$ with estimates in  H\" older spaces
 $C^{k,\alpha}\left(\h\right)$, including the degenerate case $L^\infty\left(\h\right)$. See \cite{ChaumatChollet91},  \cite{ChaumatChollet93}, and \cite{MaMichel}. In these papers, a reduction of extending the $\dbar$ data to supersets of $\h$ is allowed by the H\" older norms and used essentially. This reduction does not occur for $L^p$ data.

Our analysis has a surprising consequence: solution operators for $\dbar$ on $\h$ exist that are better behaved than the canonical solution operator in terms of $L^p$ boundedness. Recent results, \cite{EdhMcN16}, \cite{ChaZey16}, and \cite{Chen17}
show the Bergman projection on $\h$ is only $L^p$ bounded for $p\in\left(\frac 43, 4\right)$. The canonical solution operator inherits this limited $L^p$ boundedness. However our solution operator for $\dbar$ on $\h$ is $L^p$ bounded for all $1\leq p<\infty$, at least on a subclass of forms. See Sections \ref{SS:example} and \ref{SS:extra} for details. The only other situation we know where the canonical solution operator is demonstrably not the best solution operator for $\dbar$ is the Diederich-Fornaess worm domain $W$, for estimates in the $C^k\left(\overline{W}\right)$ scale of norms, c.f. \cite{Christ96} and \cite{Kohn73}.

The authors thank Dror Varolin for an insightful comment about section \ref{SS:product}. The authors are also grateful to the anonymous referee, who pointed out an error in an earlier version of Section \ref{S:necessity} and suggested several expositional improvements.


\section{The Cauchy-Fantappi\'{e} Formalism}\label{cfformalism}

In this section the Cauchy-Fantappi\'{e} formalism is reviewed and applied to product spaces. For more information, see \cite{Range86,RangeSiu}.

\subsection{Domain with $C^1$ boundary}
\label{C1domain}

Let $D$ be a domain in $\C^n$ with $C^1$ boundary $bD$. Let $U$ be an open neighborhood of $bD\times\bar D$ and $U^*=\{(\zeta,z)\in U | \zeta\ne z\}$.

\begin{definition}\label{D:generateC1domain}
A generating form $w$ on $bD\times D$ is a $C_{1,0}^1$-form in $\zeta$ and a $C^{\infty}$ function in $z$,
\[
w(\zeta,z)=\sum_{l=1}^n w^l(\zeta,z)\,d\zeta_l,
\]
with the following property
\[
\langle w(\zeta,z),\zeta-z\rangle = \sum_{l=1}^n w^l(\zeta,z)(\zeta_l-z_l)=1\qquad\text{for}\,\,\,(\zeta,z)\in U^*.
\]
\end{definition}

When applying a differential operator to forms depending on multiple sets of independent variables (like $w$), subscripts will be used to indicate which variables are differentiated. For instance, 
$\dbar_z w(\zeta, z) =\sum_{k,l=1} \frac{\partial w^l}{\partial\bar z_k} d\bar z_k\wedge d\zeta_l$, while $\partial_\zeta w(\zeta, z) =\sum_{k,l=1} \frac{\partial w^l}{\partial\zeta_k} d\zeta_k\wedge d\zeta_l$.
The same convention is used on functions.

The universal form
\[
w_0(\zeta,z)=\frac{\partial_{\zeta} (|\zeta-z|^2)}{|\zeta-z|^2}=\frac{\sum_{l=1}^n(\bar\zeta_l-\bar z_l)\,d\zeta_l}{|\zeta-z|^2}
\]
is called the Bochner-Martinelli generating form. This form satisfies Definition \ref{D:generateC1domain} for any domain $D$. Let $\mu\in I=[0,1]$. If $w$ is a generating form on $bD\times D$, define the homotopy between $w$ and $w_0$ by
\begin{equation}\label{E:homotopy}
\hat w(\zeta,z,\mu)=\mu w(\zeta,z)+(1-\mu) w_0(\zeta,z).
\end{equation}
Note that for each fixed $\mu\in I$, the form $\hat w$ also satisfies Definition \ref{D:generateC1domain}.

A piece of notation simplifies writing formulas below: let $\bar \partial_{\zeta,\mu}=\bar \partial_{\zeta}+d_{\mu}$.

\begin{definition}
\label{CFkernel}
The Cauchy-Fantappi\'{e} kernel of order $q$ generated by $\hat w$ is 
\begin{equation}\label{E:CFkernel1}
\Omega_q(\hat w)=\frac{(-1)^{q(q-1)/2}}{(2\pi i)^n}\left( \begin{array}{c} n-1 \\ q \end{array} \right) \hat w \wedge \left( \bar \partial_{\zeta,\mu} \hat w \right)^{n-q-1} \wedge \left( \bar \partial_{z} \hat w \right)^q
\end{equation}
for $0 \le q \le n-1$, and $0$ otherwise ($q=-1$ and $q=n$).
\end{definition}
\smallskip

\begin{remark}\label{R:1}
The kernel $\Omega_q(w)$ associated to an arbitrary generating form $w$ is defined in the same manner:
 \begin{equation}\label{E:CFkernel2}
\Omega_q(w)=\frac{(-1)^{q(q-1)/2}}{(2\pi i)^n}\left( \begin{array}{c} n-1 \\ q \end{array} \right)  w \wedge \left( \bar \partial_{\zeta} w \right)^{n-q-1} \wedge \left( \bar \partial_{z}  w \right)^q
\end{equation}
Note that $\bar\partial_{\zeta,\mu}$ has been replaced by $\bar\partial_{\zeta}$. Moreover, if $q\ge1$ and $w$ is holomorphic in $z$, $\Omega_q(w)=0$ because of the final factor in \eqref{E:CFkernel2}.
\end{remark}
\smallskip

\begin{remark}\label{R:2}
The kernel $\Omega_q(w_0)$ is also denoted $K_q$ and called the Bochner-Martinelli-Koppelman kernel, following \cite{Range86}. For $D\subset\subset\C^n$ with piecewise $C^1$ boundary,  the Bochner-Martinelli-Koppelman representation for $f\in C_{0,q}^1(\bar D)$ is
\begin{equation}\label{E:BMK}
f(z)=\int_{bD}f\wedge K_q(\cdot,z)-\int_D\bar\partial f\wedge K_q(\cdot,z)-\bar \partial_z\int_Df\wedge K_{q-1}(\cdot,z)
\end{equation}
where $0\le q\le n$. It holds that $\int_Df\wedge K_{q-1}(\cdot,z)\in C_{0,q-1}^1(D)$.  See \cite[Chap. IV, Theorem 1.10]{Range86} for proofs of these facts.
\end{remark}

A more general representation formula than \eqref{E:BMK} uses the following ingredient:

\begin{definition}\label{D:solutionop}
Let $D\subset\subset\C^n$ be a domain with $C^1$ boundary and $w$ a generating form on $bD\times D$.
For $1 \le q \le n$, define the integral operator
\[
T_{q}^{w}: C_{0,q}(\bar D) \to C_{0,q-1}(D)
\]
by
\[
T_{q}^{w}(f)=\int_{bD \times I}f \wedge \Omega_{q-1}(\hat w)-\int_D f \wedge K_{q-1}.
\]
Set $T_0^w=T_{n+1}^w \equiv 0$.
\end{definition}

The following theorem is proved in \cite[Chap. IV, Theorem 3.6]{Range86}.

\begin{theorem}
Let $D\subset\subset\C^n$ be a domain with $C^1$ boundary and $w$ a generating form on $bD\times D$.
For $0 \le q \le n$ and $f \in C_{0,q}^1(\bar D)$, 
\begin{equation}\label{E:SolutionOp1}
f=\int_{bD}f \wedge \Omega_q(w) + \bar \partial T_q^w(f) + T_{q+1}^w(\bar \partial f)\qquad\text{on}\,\,\, D.
\end{equation}

Moreover, for $k=0,1,2,\dots,\infty$, if $f \in C_{0,q}^k(D) \cap C_{0,q}(\bar D)$ then $T_q^w(f) \in C_{0,q-1}^k(D)$.\\
\end{theorem}

\begin{remark}\label{R:3}
Suppose $q\geq1$. If the generating form $w$ is holomorphic in $z$ and $f \in C_{0,q}^1(\bar D)$ is $\bar \partial$-closed, \eqref{E:SolutionOp1} implies $u=T_{q}^w(f)$ solves
\[
\bar \partial u=f,
\]
since $\Omega_q(w)=0$ as noted in Remark \ref{R:1}
\end{remark}

\subsection{Product domains}\label{SS:product} An idea from \cite{RangeSiu} is used to construct a generating form on a product domain from known generating forms on the factors. In \cite{RangeSiu}, only domains with piecewise smooth boundaries that are {\it strongly pseudoconvex} away from boundary singularities are considered. However, strong pseudoconvexity is only used to build the integral kernels on
smooth pieces of the boundary, not to piece the kernels together to get a solution operator for $\dbar$. This latter idea is what we extract.

Let $D=D_1 \times D_2\subset\C^n$, where $D_1\subset\C^{n_1}$ and $D_2\subset\C^{n_2}$ are domains with $C^1$ boundary. Let $S_1=bD_1\times\bar D_2$ and $S_2=\bar D_1\times bD_2$.

\begin{definition}
\label{D:generateproductdomain}
For $j=1,2$, a generating form $w_j(\zeta,z)$ on $S_j\times D$ is a $(1,0)$-form in $\zeta$
\[
w_j(\zeta,z)=\sum_{l=1}^n w_j^l(\zeta,z)\,d\zeta_l
\]
with the following properties
\begin{enumerate}
\item for each fixed $\zeta\in S_j$, $w_j^l(\zeta,\cdot)$ is $C^{\infty}$ in $D$,
\item for each fixed $z\in D$, $w_j^l(\cdot,z)$ is $C^1$ in a neighborhood $U_j^z$ of $S_j$, and
\item for each $z\in D$
\begin{equation}
\label{geniden}
\langle w_j(\zeta,z),\zeta-z\rangle = \sum_{l=1}^n w_j^l(\zeta,z)(\zeta_l-z_l)=1\qquad \text{for all} \,\,\,\zeta\in U_j^{z}.
\end{equation}
\end{enumerate}
\end{definition}

\begin{remark}\label{R:generating}
The forms in Definition \ref{D:generateproductdomain} are generating for only part of $bD$, namely $S_j$.
To connect  this with the previous definition, suppose $\tilde w_j$ is a generating form on $bD_j\times D_j$ as in Definition \ref{D:generateC1domain}, for $j=1,2$. Define
\[
w_1(\zeta,z)=\sum_{l=1}^{n_1}\tilde w_1^l(\zeta^1,z^1)\,d\zeta^1_l
\]
and
\[
w_2(\zeta,z)=\sum_{l=1}^{n_2}\tilde w_2^{l}(\zeta^2,z^2)\,d\zeta^2_l
\]
where $\zeta=\left(\zeta^1,\zeta^2\right)\in\C^{n_1}\times\C^{n_2}$ and $z=\left(z^1,z^2\right)\in\C^{n_1}\times\C^{n_2}$. Note that $w_1, w_2$ are independent of $\left(\zeta^2, z^2\right)$, $\left(\zeta^1, z^1\right)$ respectively.
Elementary algebra shows that $w_1$ and $w_2$ are generating forms on $S_1\times D$ and $S_2\times D$, respectively, as given by Definition \ref{D:generateproductdomain}. 
\end{remark}

Following  \cite{RangeSiu}, let
\[
\Delta=\{\lambda=(\lambda_0,\lambda_1,\lambda_2)\in\R^{3}\,|\,\lambda_0,\lambda_1,\lambda_2\ge0,\lambda_0+\lambda_1+\lambda_2=1\},
\]
\[
\Delta_{0}=\{\lambda\in\Delta\,|\,\lambda_1=\lambda_2=0\},
\]
\[
\Delta_{01}=\{\lambda\in\Delta\,|\,\lambda_2=0\},
\]
\[
\Delta_{02}=\{\lambda\in\Delta\,|\,\lambda_1=0\}.
\]
As before, let
\[
w_0(\zeta,z)=\frac{\partial_{\zeta} (|\zeta-z|^2)}{|\zeta-z|^2}=\frac{\sum_{l=1}^n(\bar\zeta_l-\bar z_l)\,d\zeta_l}{|\zeta-z|^2}.
\]
Consider the partial convex combination of $w_0$, $w_1$, and $w_2$
\begin{equation}
\label{cvxcmb}
W(\zeta,z,\lambda)=\sum_{j=0}^{2}\lambda_j w_j(\zeta,z),
\end{equation}
defined only on the following sets
\begin{enumerate}
\item $(\zeta,z,\lambda)\in\bar D\times D\times\Delta_0$ with $\zeta\neq z$,
\item $(\zeta,z,\lambda)\in S_1\times D\times\Delta_{01}$,
\item $(\zeta,z,\lambda)\in S_2\times D\times\Delta_{02}$,
\item and $(\zeta,z,\lambda)\in (bD_1\times bD_2)\times D\times\Delta$.
\end{enumerate}
\smallskip

\begin{remark}
For $\zeta$ fixed, the form $W$ in \eqref{cvxcmb} is $C^{\infty}$ for $z\in D$. For $z\in D$ fixed, $W$ is $C^1$ in $\zeta$ and satisfies \eqref{geniden} in the corresponding neighborhood depending on $z$. Also, the form $W$ is differentiable in $\lambda$ in the interiors of $\Delta$, $\Delta_{01}$, and $\Delta_{02}$.
\end{remark}

When derivatives with respect to the vector $\lambda$ are written, the meaning is that derivatives with respect to  $\lambda_0, \lambda_1, \lambda_2$ are taken and the results are added. Notationally
\[
\bar \partial_{\zeta,\lambda}=\bar \partial_{\zeta}+d_{\lambda_0}+d_{\lambda_1}+d_{\lambda_2}.
\]

Similar to Definition \ref{CFkernel}, a kernel is associated to the form $W$ in \eqref{cvxcmb}:

\begin{definition}
The Cauchy-Fantappi\'{e} kernel of order $q$ generated by the form $W$ in \eqref{cvxcmb} is defined 
\[
\Omega_q(W)=\frac{(-1)^{q(q-1)/2}}{(2\pi i)^n}\left( \begin{array}{c} n-1 \\ q \end{array} \right) W \wedge \left( \bar \partial_{\zeta,\lambda} W \right)^{n-q-1} \wedge \left( \bar \partial_{z} W \right)^q
\]
for $0 \le q \le n-1$, and $0$ otherwise ($q=-1$ and $q=n$).
\end{definition}

\begin{remark}\label{R:CF}
If $q\ge1$ and if for $j=1,2$ $w_j$ is holomorphic in $z$ for $\zeta$ fixed, then $\Omega_q(W)=0$ on the set where $\lambda_0=0$. This follows since $\lambda_0=0$ implies that $W$ defined by \eqref{cvxcmb} is holomorphic in $z$.

On the other hand, note that none of the sets (1)-(4) in \eqref{cvxcmb} allow $\lambda_0=0$. Nevertheless, $\Omega_q(w)=0$ on this set is needed in order to show that the operator in Definition \ref{D:solution} below is solution operator for $\dbar$; for a proof see \cite[\S(2.5)]{RangeSiu}.
\end{remark}

Parallel to \S\ref{C1domain}, an integral operator associated to the form $W$ in \eqref{cvxcmb} is defined.

\begin{definition}\label{D:solution}
For $1 \le q \le n$, define the integral operator
\begin{align*}
T_{q}^{W}: C_{0,q}(\bar D) \to C_{0,q-1}(D)
\end{align*}
by
\begin{align*}
T_{q}^{W}(f)=-\int_{bD_1\times bD_2\times \Delta}f\wedge\Omega_{q-1}(W)+\int_{S_1\times\Delta_{01}}f\wedge\Omega_{q-1}(W)&+\int_{S_2\times\Delta_{02}}f\wedge\Omega_{q-1}(W)\\ &-\int_{D\times\Delta_0} f \wedge\Omega_{q-1}(W).
\end{align*}
Set $T_0^W=T_{n+1}^W \equiv 0$.
\end{definition}

\begin{remark}\label{R:productSolution}
For $1\le q\le n$, if $w_j$ is holomorphic in $z$ for $j=1,2$, then $T_q^W$ is a solution operator to the $\bar \partial$-equations; i.e. $u=T^W_q(f)$ solves
\[
\bar \partial u=f
\]
when $\bar\partial f=0$. This follows from Stokes' theorem, but non-trivially as the different dimensional facets of the simplex $\Delta$ must be handled. As for Remark \ref{R:CF}, a detailed proof is given in \cite[\S(2.5)]{RangeSiu}.
\end{remark}

\begin{remark}
Since $\lambda_0+\lambda_1=1$ on $\Delta_{01}$,  $d\lambda_0=-d\lambda_1$ on this set. By change of variables, it follows that
\[
\int_{S_1\times\Delta_{01}}f\wedge\Omega_{q-1}(W)=\int_{S_1\times I}f\wedge\Omega_{q-1}(\hat w_1),
\]
where $\hat w_1$ is the homotopic form as in \S\ref{C1domain} and $\mu\in I$. Similarly, 
\[
\int_{S_2\times\Delta_{02}}f\wedge\Omega_{q-1}(W)=\int_{S_2\times I}f\wedge\Omega_{q-1}(\hat w_2).
\]
Moreover, since $\lambda_0=1$ on $\Delta_0$, $w=w_0$ on this singleton. Thus 
\begin{align*}
T_{q}^{W}(f)=-\int_{bD_1\times bD_2\times \Delta}f&\wedge\Omega_{q-1}(W)+\int_{bD_1\times D_2 \times I}f \wedge \Omega_{q-1}(\hat w_1)\\&+\int_{D_1\times bD_2\times I}f\wedge\Omega_{q-1}(\hat w_2) -\int_D f \wedge K_{q-1}
\end{align*}
for $f\in C_{0,q}(\bar D)$ and $1\le q\le n$.
\end{remark}

\begin{remark}
Of particular importance here, when $D_1$ and $D_2$ are 1-dimensional the first integral in the displayed equation above vanishes. I.e., 
\begin{equation*}
\int_{bD_1\times bD_2\times \Delta}f\wedge\Omega_{0}(W)=0.
\end{equation*}
This follows since the degree of the form (with respect to the integration variable) in the integrand must equal the dimension of the set over which it is integrated, otherwise the integral is  0. This argument will be called
{\it dimension-degree counting} when used below. When $D=D_1\times D_2$ is 2-dimensional, the operator $T_1^W$ reduces to 
\begin{equation}
\label{T1}
T_{1}^{W}(f)=\int_{bD_1\times D_2 \times I}f \wedge \Omega_{0}(\hat w_1)+\int_{D_1\times bD_2\times I}f\wedge\Omega_{0}(\hat w_2)-\int_D f \wedge K_{0}
\end{equation}
for $f\in C_{0,1}^1(\bar D)$, where
\[
\Omega_0(\hat w_j)=\frac{1}{(2\pi i)^2}w_j\wedge w_0\wedge d\mu\qquad\text{for}\,\,\,j=1,2
\]
and 
\[
K_0=\frac{1}{(2\pi i)^2}\frac{(\bar\zeta_1-\bar z_1)\,d\zeta_1\wedge d\bar \zeta_2 \wedge d\zeta_2+(\bar \zeta_2-\bar z_2)\,d\zeta_2\wedge d\bar \zeta_1\wedge d\zeta_1}{|\zeta-z|^4}.
\]
The form on the operator $T^W_1$ given by \eqref{T1} is the starting point for the computations in the next section.

By dimension-degree counting, it is also easy to see
\[
T_2^{W}(f)=-\int_Df\wedge K_1\qquad\text{for}\,\,\,f\in C_{0,2}^1(\bar D),
\]
where
\[
K_1=\frac{1}{(2\pi i)^2}\frac{(\bar\zeta_1-\bar z_1)\,d\zeta_2 \wedge d\zeta_1\wedge d\bar z_2+(\bar \zeta_2-\bar z_2)\,d\zeta_1\wedge d\zeta_2\wedge d\bar z_1}{|\zeta-z|^4}.
\]
\end{remark}


\section{The $\bar \partial$-equation on product spaces}

For a two-dimensional product domain, the right hand side of \eqref{T1} can be written as explicit integral operators.
This is now derived for arbitrary bounded domains $D_1, D_2\subset\C^1$ with
$C^1$ boundary, using the Cauchy generating form $w$.

\subsection{The product space $D_1\times D_2\subset\C^2$}\label{SS:product}

Definition \ref{D:generateC1domain} shows  the Cauchy kernel
\[
w=\frac{d\zeta}{\zeta-z}
\]
is a generating form for any domain $\Omega\subset\C^1$. This form is holomorphic in $z$, away from $z=\zeta$. Set
\[
w_j=\frac{d\zeta_j}{\zeta_j-z_j}\qquad\text{for}\,\,\,j=1,2,
\]
the Cauchy kernels on the two domains $D_j$, $j=1,2$. Note that Remark \ref{R:generating} shows that $w_1, w_2$ give generating forms on $S_1\times\left(D_1\times D_2\right)$ and $S_2\times\left(D_1\times D_2\right)$ respectively.
For the rest of this section $w_j $ will refer to the Cauchy forms above, and $\hat w_j$ is defined via \eqref{E:homotopy} relative to these particular $w_j$.

Direct computation from \eqref{E:CFkernel1} gives
\[
\Omega_0(\hat w_1)=\frac{1}{(2 \pi i)^2}\frac{(\bar \zeta_2 - \bar z_2)\,d\zeta_1\wedge d\zeta_2\wedge d\mu}{(\zeta_1-z_1)|\zeta-z|^2}
\]
and
\[
\Omega_0(\hat w_2)=\frac{1}{(2 \pi i)^2}\frac{(\bar \zeta_1 - \bar z_1)\,d\zeta_2\wedge d\zeta_1\wedge d\mu}{(\zeta_2-z_2)|\zeta-z|^2}.
\]
Let $f\in C_{1,0}^1(\bar D)$ and write $f=f_1d\bar \zeta_1+f_2d\bar\zeta_2$. The first term on the right hand side of \eqref{T1} becomes
\begin{equation*}
\begin{split}
\int_{bD_1\times D_2 \times I}f \wedge \Omega_{0}(\hat w_1)
&=\int_{bD_1\times D_2 \times I}f\wedge\frac{1}{(2 \pi i)^2}\frac{(\bar \zeta_2 - \bar z_2)\,d\zeta_1\wedge d\zeta_2\wedge d\mu}{(\zeta_1-z_1)|\zeta-z|^2}\\
&=\frac{1}{(2 \pi i)^2}\int_{bD_1\times D_2 \times I}\frac{f_2 \cdot (\bar \zeta_2 - \bar z_2)\,d\bar \zeta_2 \wedge d\zeta_1\wedge d\zeta_2\wedge d\mu}{(\zeta_1-z_1)|\zeta-z|^2}\\
&=\frac{1}{(2 \pi i)^2}\int_{bD_1\times D_2}\frac{f_2 \cdot (\bar \zeta_2 - \bar z_2)\,d\bar \zeta_2 \wedge d\zeta_1\wedge d\zeta_2}{(\zeta_1-z_1)|\zeta-z|^2}.
\end{split}
\end{equation*}
The second equality follows from dimension-degree counting.

Now focus on the integration in $\zeta_1$ and apply Stokes' theorem to $\zeta_1$ on $D_1\setminus D(z_1;\tau)$, where $D(z_1;\tau)$ is the disk centered at $z_1$ with radius $\tau$. This yields
\begin{equation*}
\begin{split}
\int_{bD_1\times D_2 \times I}f \wedge \Omega_{0}(\hat w_1)
&=\frac{1}{(2 \pi i)^2}\int_{D_2} \Big[\int_{D_1\setminus D(z_1;\tau)} \frac{\partial}{\partial \bar \zeta_1} \left( \frac{f_2 \cdot (\bar \zeta_2 - \bar z_2)}{(z_1-\zeta_1)|\zeta-z|^2} \right) d\bar \zeta_1 \wedge d\zeta_1\\
&+ \int_{|\zeta_1-z_1|=\tau} \frac{f_2\cdot(\bar \zeta_2-\bar z_2)}{(z_1-\zeta_1)|\zeta-z|^2}d\zeta_1 \Big]d\bar \zeta_2 \wedge d \zeta_2\\
=\frac{1}{(2 \pi i)^2}&\int_{D_2} \Big[\int_{D_1\setminus D(z_1;\tau)} \left( \frac{\partial f_2}{\partial \bar \zeta_1} \cdot \frac{ \bar \zeta_2 - \bar z_2}{(z_1-\zeta_1)|\zeta-z|^2} + \frac{f_2 \cdot (\bar \zeta_2 - \bar z_2)}{|\zeta-z|^4} \right) d\bar \zeta_1 \wedge d\zeta_1\\
&~+ \int_{|\zeta_1-z_1|=\tau} \frac{f_2\cdot(\bar \zeta_2-\bar z_2)}{(z_1-\zeta_1)|\zeta-z|^2}d\zeta_1 \Big]d\bar \zeta_2 \wedge d \zeta_2.
\end{split}
\end{equation*}

Letting $\tau\to 0^+$ gives
\begin{equation*}
\lim_{\tau\to0^+}\int_{D_1\setminus D(z_1;\tau)} \frac{\partial f_2}{\partial \bar \zeta_1} \cdot \frac{ \bar \zeta_2 - \bar z_2}{(z_1-\zeta_1)|\zeta-z|^2} d\bar \zeta_1 \wedge d\zeta_1=\int_{D_1} \frac{\partial f_2}{\partial \bar \zeta_1} \cdot \frac{ \bar \zeta_2 - \bar z_2}{(z_1-\zeta_1)|\zeta-z|^2} d\bar \zeta_1 \wedge d\zeta_1
\end{equation*}
\begin{equation*}
\lim_{\tau\to0^+}\int_{D_1\setminus D(z_1;\tau)}\frac{f_2 \cdot (\bar \zeta_2 - \bar z_2)}{|\zeta-z|^4} d\bar \zeta_1 \wedge d\zeta_1=\int_{D_1}\frac{f_2 \cdot (\bar \zeta_2 - \bar z_2)}{|\zeta-z|^4} d\bar \zeta_1 \wedge d\zeta_1,
\end{equation*}
and
\begin{equation*}
\lim_{\tau\to0^+}\int_{|\zeta_1-z_1|=\tau} \frac{f_2\cdot(\bar \zeta_2-\bar z_2)}{(z_1-\zeta_1)|\zeta-z|^2}d\zeta_1=-2\pi i \frac{f_2(z_1,\zeta_2)}{\zeta_2-z_2}.
\end{equation*}
Substituting these terms in the previous equation, the first term on the right hand side of \eqref{T1} can be written
\begin{equation*}
\begin{split}
\int_{bD_1\times D_2 \times I}f \wedge \Omega_{0}(\hat w_1)
&=\frac{-1}{2 \pi i}\int_{D_2}\frac{f_2(z_1,\zeta_2)}{\zeta_2-z_2}d\bar \zeta_2\wedge d\zeta_2 \\&+\frac{1}{(2 \pi i)^2}\int_{D_1 \times D_2}\frac{f_2\cdot(\bar \zeta_2 - \bar z_2)}{|\zeta-z|^4} d\bar \zeta_1 \wedge d\zeta_1\wedge d \bar \zeta_2\wedge d \zeta_2 \\
&+\frac{-1}{(2 \pi i)^2}\int_{D_1 \times D_2}\frac{\partial f_2}{\partial \bar \zeta_1}\cdot\frac{\bar \zeta_2-\bar z_2}{(\zeta_1-z_1)|\zeta-z|^2}\,d\bar \zeta_1\wedge d\zeta_1\wedge d\bar \zeta_2\wedge d\zeta_2.
\end{split}
\end{equation*}
Similarly, the second term on the right hand side in \eqref{T1} can be written
\begin{equation*}
\begin{split}
\int_{D_1\times bD_2\times I}f\wedge\Omega_{0}(\hat w_2)
&=\frac{-1}{2 \pi i}\int_{D_1}\frac{f_1(\zeta_1,z_2)}{\zeta_1-z_1}d\bar \zeta_1\wedge d\zeta_1\\&+\frac{1}{(2 \pi i)^2}\int_{D_1 \times D_2}\frac{f_1\cdot(\bar \zeta_1 - \bar z_1)}{|\zeta-z|^4} d\bar \zeta_1 \wedge d\zeta_1\wedge d \bar \zeta_2\wedge d \zeta_2\\
&+\frac{-1}{(2 \pi i)^2}\int_{D_1 \times D_2}\frac{\partial f_1}{\partial \bar \zeta_2}\cdot\frac{\bar \zeta_1-\bar z_1}{(\zeta_2-z_2)|\zeta-z|^2}\,d\bar \zeta_1\wedge d\zeta_1\wedge d\bar \zeta_2\wedge d\zeta_2.
\end{split}
\end{equation*}
Note that the last term on the right hand side of \eqref{T1} is
\begin{equation*}
\int_{D_1\times D_2} f \wedge K_{0}=\frac{1}{(2 \pi i)^2}\int_{D_1\times D_2}\frac{f_1\cdot(\bar \zeta_1 - \bar z_1)+f_2\cdot(\bar \zeta_2-\bar z_2)}{|\zeta-z|^4} d\bar \zeta_1\wedge d\zeta_1\wedge d\bar \zeta_2\wedge d\zeta_2.
\end{equation*}

Hence, formula \eqref{T1} can be expressed 
\begin{equation}
\label{4termT1}
\begin{split}
T^W_1(f)
&=\int_{bD_1\times D_2 \times I}f \wedge \Omega_{0}(\hat w_1)+\int_{D_1\times bD_2\times I}f\wedge\Omega_{0}(\hat w_2)-\int_{D_1\times D_2} f \wedge K_{0}\\
&=\frac{-1}{2 \pi i}\int_{D_2}\frac{f_2(z_1,\zeta_2)}{\zeta_2-z_2}d\bar \zeta_2\wedge d\zeta_2+\frac{-1}{2 \pi i}\int_{D_1}\frac{f_1(\zeta_1,z_2)}{\zeta_1-z_1}d\bar \zeta_1\wedge d\zeta_1\\
&~+\frac{-1}{(2 \pi i)^2}\int_{D_1 \times D_2}\frac{\partial f_2}{\partial \bar \zeta_1}\cdot\frac{\bar \zeta_2-\bar z_2}{(\zeta_1-z_1)|\zeta-z|^2}\,d\bar \zeta_1\wedge d\zeta_1\wedge d\bar \zeta_2\wedge d\zeta_2\\
&~+\frac{-1}{(2 \pi i)^2}\int_{D_1 \times D_2}\frac{\partial f_1}{\partial \bar \zeta_2}\cdot\frac{\bar \zeta_1-\bar z_1}{(\zeta_2-z_2)|\zeta-z|^2}\,d\bar \zeta_1\wedge d\zeta_1\wedge d\bar \zeta_2\wedge d\zeta_2.
\end{split}
\end{equation}

\begin{definition}\label{D:D} If $f=f_1d\bar z_1+f_2d\bar z_2$ is a $(0,1)$-form, let $$\sd f= \frac 12\left( \frac{\partial f_1}{\partial\bar z_2} +  \frac{\partial f_2}{\partial\bar z_1}\right),$$
with derivatives taken in the distributional sense.
If $\bar\partial f=0$, note $\sd f =\frac{\partial f_1}{\partial \bar z_2}=\frac{\partial f_2}{\partial \bar z_1}$. 
\end{definition}

Using Definition \ref{D:D}, combine the last two terms in \eqref{4termT1}. The following expression for a strong solution operator on $D_1\times D_2$ is obtained:

\begin{proposition}\label{P:strong} Suppose $D_1, D_2\subset\C$ are domains with $C^1$ boundary. If $f\in C^1_{0,1}\left(\overline{D_1\times D_2}\right)$ satisfies $\dbar f=0$, define
\begin{equation}
\label{derivativeT1}
\begin{split}
T(f)
&=\frac{-1}{2 \pi i}\int_{D_2}\frac{f_2(z_1,\zeta_2)}{\zeta_2-z_2}d\bar \zeta_2\wedge d\zeta_2+\frac{-1}{2 \pi i}\int_{D_1}\frac{f_1(\zeta_1,z_2)}{\zeta_1-z_1}d\bar \zeta_1\wedge d\zeta_1\\
&~+\frac{-1}{(2 \pi i)^2}\int_{D_1 \times D_2}\frac{\sd(f)(\zeta_1,\zeta_2)}{(\zeta_1-z_1)(\zeta_2-z_2)}\,d\bar \zeta_1\wedge d\zeta_1\wedge d\bar \zeta_2\wedge d\zeta_2.
\end{split}
\end{equation}

Then $\dbar (Tf) =f$.
\end{proposition}

\begin{remark}
\label{T1forDxA}
Consider the third term on the right hand side of \eqref{4termT1}. If the idea from \cite{FornaessLeeZhang, Henkin71} is followed and Stokes' theorem is applied  in the $\zeta_2$ variable, this term becomes
\begin{align*}
\int_{D_1\times D_2}\frac{\partial f_1}{\partial \bar \zeta_2} \cdot \frac{ \bar \zeta_2 - \bar z_2}{(\zeta_1-z_1)|\zeta-z|^2} d\bar \zeta_1 \wedge d\zeta_1\wedge d \bar \zeta_2 \wedge d\zeta_2
&=\int_{D_1} \Big[ \int_{bD_2} \frac{f_1\cdot(\bar \zeta_2-\bar z_2)}{(\zeta_1-z_1)|\zeta-z|^2} d\zeta_2\\
&~-\int_{D_2} \frac{f_1\cdot(\bar \zeta_1-\bar z_1)}{|\zeta-z|^4}d\bar \zeta_2\wedge d\zeta_2 \Big] d\bar \zeta_1\wedge d\zeta_1.
\end{align*}

Similarly, the last term in \eqref{4termT1} can be rewritten as
\begin{align*}
\int_{D_1\times D_2}\frac{\partial f_2}{\partial \bar \zeta_1} \cdot \frac{ \bar \zeta_1 - \bar z_1}{(\zeta_2-z_2)|\zeta-z|^2} d\bar \zeta_1 \wedge d\zeta_1\wedge d \bar \zeta_2 \wedge d\zeta_2
&=\int_{D_2} \Big[ \int_{bD_1} \frac{f_2\cdot(\bar \zeta_1-\bar z_1)}{(\zeta_2-z_2)|\zeta-z|^2} d\zeta_1\\
&-\int_{D_1} \frac{f_2\cdot(\bar \zeta_2-\bar z_2)}{|\zeta-z|^4}d\bar \zeta_1\wedge d\zeta_1 \Big] d\bar \zeta_2\wedge d\zeta_2.
\end{align*}

Thus an alternative expression for the operator $T=T_1^W$ is
\begin{equation*}
\begin{split}
T(f)
&=\frac{-1}{2 \pi i}\int_{D_2}\frac{f_2(z_1,\zeta_2)}{\zeta_2-z_2}d\bar \zeta_2\wedge d\zeta_2-\frac{1}{(2 \pi i)^2}\int_{D_1\times bD_2}\frac{f_1\cdot(\bar \zeta_2-\bar z_2)}{(\zeta_1-z_1)|\zeta-z|^2} d\bar \zeta_1\wedge d\zeta_1\wedge d\zeta_2\\
&~+\frac{-1}{2 \pi i}\int_{D_1}\frac{f_1(\zeta_1,z_2)}{\zeta_1-z_1}d\bar \zeta_1\wedge d\zeta_1-\frac{1}{(2 \pi i)^2}\int_{bD_1\times D_2}\frac{f_2\cdot(\bar \zeta_1-\bar z_1)}{(\zeta_2-z_2)|\zeta-z|^2} d\zeta_1\wedge d \bar \zeta_2 \wedge d\zeta_2\\
&+\frac{1}{(2 \pi i)^2}\int_{D_1\times D_2}\frac{f_1\cdot(\bar \zeta_1 - \bar z_1)+f_2\cdot(\bar \zeta_2-\bar z_2)}{|\zeta-z|^4} d\bar \zeta_1\wedge d\zeta_1\wedge d\bar \zeta_2\wedge d\zeta_2.
\end{split}
\end{equation*}
This formula corrects a small error in \cite{Henkin71, FornaessLeeZhang}, where a different constant appears before the last term.
\end{remark}


\section{The $L^p$ estimate of the solution operator}
\label{LpofT}

For the rest of the paper, $T$ denotes the operator defined by \eqref{derivativeT1}.
\subsection{The $L^p$ estimate of $T$}

As a integral operator, $T$ is first shown to be well-defined and bounded between particular Banach spaces. The following lemma is used.

\begin{lemma}
\label{lem135}
Let $D_1, D_2\subset\C$ be bounded domains and $1\le p<\infty$. If $g\in L^p(D_1\times D_2)$, the functions
\[
\frac{-1}{2 \pi i}\int_{D_1}\frac{g(\zeta_1, z_2)}{\zeta_1-z_1}d\bar \zeta_1\wedge d\zeta_1\quad\text{and}\quad\frac{-1}{2 \pi i}\int_{D_2}\frac{g(z_1,\zeta_2)}{\zeta_2-z_2}d\bar \zeta_2\wedge d\zeta_2
\]
belong to $L^p(D_1\times D_2)$. Their $L^p$-norms are bounded by $C\|g\|_{L^p(D_1\times D_2)}$, for a constant $C>0$ independent of $g$.

\end{lemma}

\begin{proof} The symmetry of the functions show it suffices to prove the result for either one; consider the second function.
Let $g^{z_1}(\zeta_2)=g(z_1,\zeta_2)\chi_{D_2}(\zeta_2)$, where $\chi_{D_2}$ is the characteristic function over $D_2$. Since $g\in L^p(D_1\times D_2)$,  $g^{z_1}\in L^p(\C)$.

Let $B=B(0;R)$ be the disk centered at $0$ of radius $R$ in $\C$ and $h(\zeta)=\frac{1}{|\zeta|}\chi_{B}(\zeta)$, where $\chi_{B}$ is the characteristic function over $B$ and $R$ is sufficiently large (say $R>\text{diam}(D_2)$). Then $h\in L^1(\C)$.

By Young's inequality, $g^{z_1}*h\in L^p(\C)$. Note that, for any $z_2,\zeta_2\in D_2$, $|z_2-\zeta_2|\le\text{diam}(D_2)<R$, so $z_2-\zeta_2\in B$. Therefore,
\begin{align*}
\int_{D_2}\left|\int_{D_2}\frac{g(z_1,\zeta_2)\,dA(\zeta_2)}{\zeta_2-z_2}\right|^p\,dA(z_2)
&\le\int_{\C}\left(\int_{\C}\frac{|g^{z_1}(\zeta_2)|\chi_B(\zeta_2-z_2)\,dA(\zeta_2)}{|\zeta_2-z_2|}\right)^p\,dA(z_2)\\
&\le C\int_{D_2}|g(z_1,z_2)|^p\,dA(z_2).
\end{align*}
The conclusion follows by integrating this inequality in $z_1$ over $D_1$. 
\end{proof}

Lemma \ref{lem135} applies to the first two terms on the right hand side of \eqref{derivativeT1}. Thus, the $L^p$-norms of these terms are bounded by $C\|f_2\|_{L^p(D_1\times D_2)}$ and $C\|f_1\|_{L^p(D_1\times D_2)}$ respectively, provided $f_1,f_2\in L^p(D_1\times D_2)$. Additionally, if $\sd(f)\in L^p(D_1\times D_2)$, two applications of Lemma \ref{lem135} show the last term in \eqref{derivativeT1} is in $L^p(D_1\times D_2)$, with $L^p$-norm bounded by $C\|\sd(f)\|_{L^p(D_1\times D_2)}$.

\begin{definition}\label{D:banach} For $1\le p<\infty$, define the Banach space of $(0,1)$-forms on $D_1\times D_2$ 
\[
\cb=\left\{f=f_1d\bar z_1+f_2d\bar z_2~|~f_1,f_2\in L^p\,\text{and}\, \sd(f)\in L^p\right\}
\]
with norm
$\|f\|_{\cb}:=\|f_1\|_{L^p(D_1\times D_2)}+\|f_2\|_{L^p(D_1\times D_2)}+\|\sd(f)\|_{L^p(D_1\times D_2)}$.
\end{definition}
The argument above Definition \ref{D:banach} proves

\begin{lemma}
\label{lem24}

The operator $T$ given by \eqref{derivativeT1} maps $\cb$ to $L^p(D_1\times D_2)$ boundedly. 
\end{lemma}

\subsection{Passage to a weak solution}\label{SS:punctured}

Lemma \ref{lem24} shows that $T:\cb\to L^p(D_1\times D_2)$ is well-defined. By \eqref{derivativeT1}, $T$ is a strong solution operator to $\dbar(Tf)=f$ if $T$ is restricted to $f\in C^1_{0,1}\left(\overline{D_1\times D_2}\right)$ and $\dbar f=0$. In this section $Tf$ is shown to be a weak solution to $\dbar(Tf)=f$, if $\dbar f=0$ weakly and $f\in\cb$ , by a limit argument.

\begin{theorem}
\label{Lpwithdbarf}
For $j=1,2$, let $D_j\in\C$ be bounded domains with $C^1$ boundary. Let $f$ be a $(0,1)$-form that is $\dbar$-closed in the weak sense on $D_1\times D_2$. For $1\le p<\infty$, assume that $f\in\cb$. 

Then $u=Tf$, defined by \eqref{derivativeT1},  is a weak solution to the equation $\dbar u=f$ on $D_1\times D_2$ and satisfies the estimate
\[
\|T(f)\|_{L^p(D_1\times D_2)}\le C \|f\|_\cb
\]
for a constant $C>0$ independent of $f$.
\end{theorem}

\begin{proof}
For $j=1,2$, let $\rho_j$ be a defining function for the domain $D_j$. Let $D_j^{\delta}=\{z\in\C\,|\,\rho_j(z)<-\delta\}$ for $\delta>0$ sufficiently small. Denote by $T^{\delta}$ the operator in \eqref{derivativeT1} with $D_j$ replaced by $D_j^{\delta}$. Then
\[
\dbar T^{\delta}(f)=f\qquad\text{on\,\,}D_1^{\delta}\times D_2^{\delta}\qquad\text{for\,\,}\dbar\text{-closed\,\,}(0,1)\text{-form\,\,}f\in C^{1}(\overline{D_1^{\delta}\times D_2^{\delta}}).
\]

Let $1\le p<\infty$. For $f\in\cb$, the standard mollifier argument (see for example \cite[Chap. 5.3, Theorem 2]{Evans98}) gives a sequence $\{f^{\ve_j}\}\subset C^{1}(\overline{D_1^{\delta}\times D_2^{\delta}})$, so that
\[
\dbar f^{\ve_j}=0,
\]
\[
f^{\ve_j}\to f\,\,\,\,\text{and}\,\,\,\,\sd(f^{\ve_j})\to \sd(f)\,\,\,\,\text{in}\,\,L^p(D_1^{\delta}\times D_2^{\delta})\,\,\,\,\text{as}\,\,\ve_j\to0.
\]
Thus
\[
\dbar T^{\delta}(f^{\ve_j})=f^{\ve_j}\qquad\text{on\,\,}D_1^{\delta}\times D_2^{\delta}
\]
in the strong sense.

On the other hand, replacing $D_1\times D_2$ by $D_1^{\delta}\times D_2^{\delta}$ in Lemma \ref{lem24} and denoting the Banach space on $D_1^{\delta}\times D_2^{\delta}$ by $\cb^{\delta}$, it follows that $T^{\delta}$ is bounded from $\cb^{\delta}$ to $L^p(D_1^{\delta}\times D_2^{\delta})$. So for $f\in\cb$,  $\lim T^{\delta}(f^{\ve_j})=T^{\delta}(f)$ in $L^p(D_1^{\delta}\times D_2^{\delta})$ as $\ve_j\to0$. Hence $T^{\delta}$ weakly solves the $\dbar$-equation on $D_1^{\delta}\times D_2^{\delta}$.

Next, for each $f\in\cb$, extend $T^{\delta}(f)$ to a function on $D_1\times D_2$ by setting it equal $0$ outside $D_1^{\delta}\times D_2^{\delta}$. Consider $\|T(f)-T^{\delta}(f)\|_{L^p(D_1\times D_2)}$.
Note that $f_2$ can be replaced by $f_2\cdot\chi_{D_2\setminus D_2^{\delta}}(\zeta_2)$ in Lemma \ref{lem135}, where $\chi_{D_2\setminus D_2^{\delta}}$ is the characteristic function over $D_2\setminus D_2^{\delta}$. Thus,  the $L^p$-norm in the conclusion in Lemma \ref{lem135} is be bounded by $C\|f_2\|_{L^p(D_1\times(D_2\setminus D_2^{\delta}))}$, which tends to $0$ as $\delta\to0^+$. A similar argument holds for Lemma \ref{lem24}. Therefore, $\lim T^{\delta}(f)=T(f)$ in $L^p(D_1\times D_2)$ as $\delta\to0^+$. This argument shows that the limit $T(f)$ is unique, and is independent of the defining functions for $D_1$ and $D_2$ used.

To show $T$ weakly solves the $\dbar$-equation on $D_1\times D_2$, argue as follows. For any $\phi\in C_c^{\infty}(D_1\times D_2)$, there is a $\delta_0>0$ so that $\text{supp}\phi\subset D_1^{\delta_0}\times D_2^{\delta_0}$. Let $K=D_1^{\delta_0}\times D_2^{\delta_0}$. Then
\[
(\dbar T(f),\phi)_{D_1\times D_2}=(\dbar T(f),\phi)_K=(T(f),\bar\partial^*\phi)_K=\lim_{\delta\to0^+}(T^{\delta}(f),\bar\partial^*\phi)_K
\]
by $L^p(D_1\times D_2)$-norm convergence. Note that $T^{\delta}$ weakly solves the $\dbar$-equation on $D_1^{\delta}\times D_2^{\delta}$. Thus for $0<\delta<\delta_0$, 
\[
\lim_{\delta\to0^+}(T^{\delta}(f),\bar\partial^*\phi)_K=\lim_{\delta\to0^+}(\dbar T^{\delta}(f),\phi)_K=(f,\phi)_K=(f,\phi)_{D_1\times D_2}.
\]
\end{proof}

Now let $D_1\times D_2=\D\times A$, where $A=A(0;1,\delta)=\{z\in\C\,|\,\delta<|z|<1\}$. Theorem \ref{Lpwithdbarf} directly applies to $\D\times A$.  However the proof of Theorem \ref{Lpwithdbarf} also applies, allowing the limit $\delta\to0^+$ to be taken. This yields the following result.

\begin{corollary}\label{C:mainProduct}
Let $f=f_1d\bar z_1+f_2d\bar z_2$ be a $(0,1)$-form that is $\bar\partial$-closed in the weak sense on $\D\times\D^*$. For $1\le p<\infty$, assume that  $f\in \cb$. 

Then $Tf$ defined
\begin{equation}
\label{T1forDxD*}
\begin{split}
T(f)
&=\frac{-1}{2 \pi i}\int_{\D^*}\frac{f_2(z_1,\zeta_2)}{\zeta_2-z_2}d\bar \zeta_2\wedge d\zeta_2+\frac{-1}{2 \pi i}\int_{\D}\frac{f_1(\zeta_1,z_2)}{\zeta_1-z_1}d\bar \zeta_1\wedge d\zeta_1\\
&~+\frac{-1}{(2 \pi i)^2}\int_{\D\times\D^*}\frac{\sd(f)(\zeta_1,\zeta_2)}{(\zeta_1-z_1)(\zeta_2-z_2)}\,d\bar \zeta_1\wedge d\zeta_1\wedge d\bar \zeta_2\wedge d\zeta_2
\end{split}
\end{equation}
is well-defined, weakly solves $\bar\partial(Tf)=f$ on $\D\times\D^*$, and satisfies the estimate
\[
\|T(f)\|_{L^p(\D\times \D^*)}\le C \|f\|_\cb
\]
for a constant $C>0$ independent of $f$.
\end{corollary}


\section{$f\in L^p$ is not sufficient for existence of $Tf$}\label{S:necessity}
The assumption $\left\|\sd(f)\right\|_p<\infty$ is part of the hypotheses in Theorem \ref{Lpwithdbarf} and Corollary \ref{C:mainProduct} via the condition $f\in\cb$. Under this hypothesis and $f\in L^p(D_1\times D_2)$,
these results imply $Tf$ exists and belongs to $L^p$.

In this section we show that only assuming $f\in L^p(D_1\times D_2)$ is not enough to conclude that $Tf\in L^p(D_1\times D_2)$ for $1\leq p <2$ in general. Thus on product domains, estimates on the data beyond $f\in L^p$
are generally needed for the Henkin solution to belong to $L^p$, unlike the situation for the Henkin solution on strongly pseudoconvex domains,  \cite{Kerzman71, Ovrelid}. 

 Consider the $\dbar$-equation on the bidisc $\D^2$. For each $k=1,2,\dots$, let $f^k=f^k_1d\bar z_1+f^k_2d\bar z_2$ on $\D^2$, where
\[
\begin{array}{ccc}
\displaystyle f^k_1(z_1,z_2)=\bar z_1^{k-1}z^k_1 \bar z_2^kz_2^k & \text{and} & \displaystyle f^k_2(z_1,z_2)=\bar z_1^kz_1^k \bar z_2^{k-1}z_2^k.
\end{array}
\]
Note each $f^k$ is $\bar\partial$-closed. Moreover, direct computation shows
\begin{align}\label{E:fL1}
\|f_1^k\|_{L^1(\D^2)}=\int_{\D^2}|\bar z_1^{k-1}z^k_1 \bar z_2^kz_2^k|\,dV(z)&=O\left(\frac{1}{k^2}\right) \notag\\
\|f_2^k\|_{L^1(\D^2)}=\int_{\D^2}|\bar z_1^kz_1^k \bar z_2^{k-1}z_2^k|\,dV(z)&=O\left(\frac{1}{k^2}\right).
\end{align}

An elementary calculation will be used to compute $T(f^k)$.

\begin{lemma}
\label{Cauchytrans}
For $z\in\D$ and $k\in\Z^+$, 
\[
\int_{\D}\frac{\bar\zeta^{k-1}\zeta^k\,d\bar\zeta\wedge d\zeta}{\zeta-z}=\frac{2\pi i}{k}(1-\bar z^k z^k).
\]
\end{lemma}

\begin{proof} This follows from the generalized Cauchy Integral formula. Details are provided for completeness.

Let  $\omega=\frac{1}{k}\cdot\frac{\bar\zeta^k\zeta^k}{\zeta-z}\,d\zeta\text{ on }\D\setminus B$,
where $B=B(z;\ve)$ is a disk centered at $z$ of radius $\ve$ sufficiently small so that $B\subset\D$. By Stokes' theorem, 
$\int_{\D\setminus B}d\omega=\int_{b\D}\omega-\int_{bB}\omega$. Since
$d\omega=(\partial+\dbar)\omega=\frac{\bar\zeta^{k-1}\zeta^k}{\zeta-z}\,d\bar\zeta\wedge d\zeta$,
it follows
\begin{align*}
\int_{\D}\frac{\bar\zeta^{k-1}\zeta^k\,d\bar\zeta\wedge d\zeta}{\zeta-z}
=\lim_{\ve\to0^+}\int_{\D\setminus B}d\omega
=\lim_{\ve\to0^+}\left(\int_{b\D}\frac{1}{k}\cdot\frac{\bar\zeta^k\zeta^k}{\zeta-z}\,d\zeta-\int_{bB}\frac{1}{k}\cdot\frac{\bar\zeta^k\zeta^k}{\zeta-z}\,d\zeta\right).
\end{align*}

By the Cauchy integral formula, the first term on the right hand side is
\[
\int_{b\D}\frac{1}{k}\cdot\frac{\bar\zeta^k\zeta^k}{\zeta-z}\,d\zeta=\frac{1}{k}\int_{b\D}\frac{d\zeta}{\zeta-z}=\frac{2\pi i}{k}.
\]
Writing $\zeta=z+\ve e^{i\theta}$ on $bB$, the second term becomes
\[
\lim_{\ve\to0^+}\int_{bB}\frac{1}{k}\cdot\frac{\bar\zeta^k\zeta^k}{\zeta-z}\,d\zeta=\lim_{\ve\to0^+}\int_0^{2\pi}\frac{1}{k}\cdot\frac{|z+\ve e^{i\theta}|^{2k}\ve e^{i\theta}i\,d\theta}{\ve e^{i\theta}}=\frac{2\pi i}{k}|z|^{2k}.
\]
The conclusion follows by combining these terms.
\end{proof}

Compute $T(f^k)$ using the explicit expression $\eqref{derivativeT1}$. For the first term, 
\begin{align*}
\frac{-1}{2\pi i}\int_{\D}\frac{f^k_2(z_1,\zeta_2)\,d\bar\zeta_2\wedge d\zeta_2}{\zeta_2-z_2}
&=\frac{-1}{2\pi i}\int_{\D}\frac{\bar z_1^k z_1^k \bar \zeta_2^{k-1} \zeta_2^k\, d\bar \zeta_2\wedge d\zeta_2}{\zeta_2-z_2}
=\frac{-1}{2\pi i}\cdot\bar z_1^k z_1^k\cdot\int_{\D}\frac{\bar\zeta_2^{k-1}\zeta_2^k\,d\bar\zeta_2\wedge d\zeta_2}{\zeta_2-z_2}\\
&=\frac{-1}{2\pi i}\cdot\bar z_1^k z_1^k\cdot\frac{2\pi i}{k}(1-\bar z_2^k z_2^k) =\frac{1}{k}|z_1z_2|^{2k}-\frac{1}{k}|z_1|^{2k}.
\end{align*}
The third equality follows from Lemma \ref{Cauchytrans}. Similarly, the second term on the right hand side in \eqref{derivativeT1} is
\begin{equation*}
\frac{-1}{2\pi i}\int_{\D}\frac{f^k_1(\zeta_1,z_2)\,d\bar\zeta_1\wedge d\zeta_1}{\zeta_1-z_1}=\frac{1}{k}|z_1z_2|^{2k}-\frac{1}{k}|z_2|^{2k}.
\end{equation*}
For the last term in \eqref{derivativeT1}, separate the variables in the integral and apply Lemma \ref{Cauchytrans} twice to get
\begin{align*}
\frac{-1}{(2\pi i)^2}\int_{\D^2}\frac{\sd(f^k)(\zeta_1,\zeta_2)\,d\bar\zeta_1\wedge d\zeta_1\wedge d\bar\zeta_2\wedge d\zeta_2}{(\zeta_1-z_1)(\zeta_2-z_2)}
&=\frac{-1}{(2\pi i)^2}\int_{\D^2}\frac{k\bar \zeta_1^{k-1}\zeta_1^k \bar \zeta_2^{k-1}\zeta_2^k\,d\bar\zeta_1\wedge d\zeta_1\wedge d\bar\zeta_2\wedge d\zeta_2}{(\zeta_1-z_1)(\zeta_2-z_2)}\\
&=\frac{-k}{(2\pi i)^2}\int_{\D}\frac{\bar\zeta_1^{k-1}\zeta_1^k\,d\bar\zeta_1\wedge d\zeta_1}{\zeta_1-z_1}\int_{\D}\frac{\bar\zeta_2^{k-1}\zeta_2^k\,d\bar\zeta_2\wedge d\zeta_2}{\zeta_2-z_2}\\
&=\frac{-k}{(2\pi i)^2}\cdot\frac{2\pi i}{k}(1-\bar z_1^k z_1^k)\cdot\frac{2\pi i}{k}(1-\bar z_2^k z_2^k)\\
&=-\frac{1}{k}+\frac{1}{k}|z_1|^{2k}+\frac{1}{k}|z_2|^{2k}-\frac{1}{k}|z_1z_2|^{2k}.
\end{align*}
Therefore,
\begin{equation}\label{E:TfnotL1}
T(f^k)=\frac{1}{k}|z_1z_2|^{2k}-\frac{1}{k}.
\end{equation}

Now define, for each $L\in\Z^+$,  $g^L=g^L_1d\bar z_1+g^L_2d\bar z_2$ with $g^L_1=\sum_{k=1}^L f^k_1$ and $g^L_2=\sum_{k=1}^L f^k_2$. Clearly each $\dbar g^L=0$. Since $\sum \frac 1{k^2}<\infty$, \eqref{E:fL1} implies $g^L\in L^1_{0,1}\left(\D^2\right)$ with $L^1$ norm bounded independent of $L$.
However \eqref{E:TfnotL1} implies
\begin{align*}
T\left(g^L\right)(z_1,z_2)= \sum_{k=1}^L T\left(f^k\right) 
=\frac{|z_1z_2|\left(1-|z_1z_2|^{2L}\right)}{1-|z_1z_2|^2}-\sum_{k=1}^L\frac{1}{k}.
\end{align*}
If $K\subset \D^2$ is a compact set and  $(z_1,z_2)\in K$, the last expression tends to $-\infty$ as $L\to\infty$, by divergence of the harmonic series. 
Thus there does not exist a constant $C$ such that $\left\|T\left(g^L\right)\right\|_{L^1}\leq C\left\|g^L\right\|_{L^1}$ for all $L$.

\begin{remark} 
Taking the full sums,  $g=g_1d\bar z_1+g_2d\bar z_2$ with $g_1=\sum_{k=1}^\infty f^k_1$ and $g_2=\sum_{k=1}^\infty f^k_2$,  gives an example where $Tg$ does not even exist. In this case, \eqref{E:fL1} still shows $g\in L^1_{0,1}\left(\D^2\right)$, while the analogue of the above computation yields
$$T(g)(z_1,z_2)=-\ln(1-|z_1z_2|^2)-\sum_{k=1}^{\infty}\frac{1}{k}\equiv\infty.$$
\end{remark}

\begin{remark}
A careful inspection of the integrals shows that $g\in L^p\left(\D^2\right)$ for $1\leq p <2$; details are left to the interested reader. Thus for the $L^1$ problem, ``over-prescribing'' integrability by requiring $g\in L^p$ for $p<2$ is still not sufficient to guarantee $Tg\in L^1$.
\end{remark}


\section{Non-canonical Solution}

The solution $u=T(f)$ in \eqref{derivativeT1} is compared with the $L^2$-minimal solution $u_{\text{can}}$ on $D_1\times D_2$. The first observation is that $u=T(f)\neq u_{\text{can}}$  on $\D^2$.

Let $h\in C^1(\overline{\D})$ be a holomorphic function on $\D$ and let
\begin{equation}
\label{noncansol}
f=z_1^kh(z_2)d\bar z_1
\end{equation}
for some positive integer $k$. It is easily checked that $\dbar f=0$. Since $f_2=0$, \eqref{derivativeT1} becomes
\begin{equation*}
u=T(f)=\frac{-1}{2\pi i}\int_{\D}\frac{f_1(\zeta_1,z_2)\,d\bar\zeta_1\wedge d\zeta_1}{\zeta_1-z_1}=\frac{-1}{2\pi i}\cdot h(z_2)\int_{\D}\frac{\zeta_1^k\,d\bar\zeta_1\wedge d\zeta_1}{\zeta_1-z_1}.
\end{equation*}

Let $\omega=\zeta_1^k\bar\zeta_1/(\zeta_1-z_1)$ and apply Stokes theorem to the integral
$\int_{\D\setminus B}d\omega$,
where $B=B(z_1,\ve)$ is the disk centered at $z_1$ of radius $\ve$ for some $\ve>0$ sufficiently small. It follows that 
\begin{equation}
\label{T1example}
\frac{1}{2\pi i}\int_{\D}\frac{\zeta_1^k\,d\bar\zeta_1\wedge d\zeta_1}{\zeta_1-z_1}=\lim_{\ve\to0^+}\left(\frac{1}{2\pi i}\int_{b\D}\frac{\zeta_1^k\bar\zeta_1\,d\zeta_1}{\zeta_1-z_1}-\frac{1}{2\pi i}\int_{bB}\frac{\zeta_1^k\bar\zeta_1\,d\zeta_1}{\zeta_1-z_1}\right).
\end{equation}
By the Cauchy integral formula, the first term in \eqref{T1example} is $z_1^{k-1}$. For the second term of \eqref{T1example}, write $\zeta_1=z_1+\ve e^{i\theta}$ and note
\[
\lim_{\ve\to0^+}\frac{1}{2\pi i}\int_{bB}\frac{\zeta_1^k\bar\zeta_1\,d\zeta_1}{\zeta_1-z_1}=z_1^k\bar z_1.
\]
Thus, $u=h(z_2)\left(z_1^k\bar z_1-z_1^{k-1}\right)$.

But the $L^2$-minimal solution of $\bar\partial u=f$ is
\begin{equation*}
u_{\text{can}}=h(z_2)\left(z_1^k\bar z_1-\frac{k}{k+1}z_1^{k-1}\right),
\end{equation*}
since it is easy to verify
\[
\langle u_{\text{can}}\,,\,z_1^mz_2^n\rangle=\int_{\D\times\D}z_1^k\bar z_1h(z_2)\bar z_1^m\bar z_2^n\,dV(z)-\frac{k}{k+1}\int_{\D\times\D}z_1^{k-1}h(z_2)\bar z_1^m\bar z_2^n\,dV(z)=0
\]
for all integers $m,n\ge0$.

\begin{remark}
Let $D_1$ and $D_2$ be bounded simply connected planar domains. Then $D_1\times D_2$ is biholomorphic to $\D^2$ under a mapping $\psi_1\otimes\psi_2$, where $\psi_j$ is the biholomorphism from $D_j$ to $\D$ for $j=1,2$. Now consider the form with the same expression as in \eqref{noncansol} and let $h\equiv 1$, that is
$f=z_1^kd\bar z_1$
on $D_1\times D_2$, for some positive integer $k$. The solution operator $T$ gives the same solution
$u=z_1^k\bar z_1-z_1^{k-1}$
as above.

However, to obtain the $L^2$-minimal solution, one must transfer the orthonormal basis on $A^2(\D^2)$ to one on $A^2(D_1\times D_2)$. Thus the expression of $u_{\text{can}}$ necessarily involves $\psi'_1$ and $\psi'_2$, unlike the expression of $u=T(f)$.
\end{remark}


\section{The Hartogs Triangle}\label{S:Hartogs}

In this section, consider the $\dbar$-equation on the Hartogs triangle $\h$:
\begin{equation*}
\h=\left\{(z_1, z_2)\in\C^2: \left|z_1\right| <\left|z_2\right| <1\right\}.
\end{equation*} 
The first step is to transfer the equation on $\h$ to the product space $\D\times\D^*$.
\subsection{Transform the $\dbar$-equation}

Let 
\[
\phi: \h\to\D\times\D^*
\]
\[
\phi(z_1,z_2)=(z_1/z_2,z_2)=(w_1,w_2)
\]
be the usual biholomorphism. Consider
\begin{equation}
\label{dbaronh}
\dbar_zv=\alpha=\alpha_1\,d\bar z_1+\alpha_2\,d\bar z_2
\end{equation}
on $\h$, where $\alpha$ is $\dbar$-closed. Using the chain rule
\[
\left\{
\begin{array}{r}
d\bar z_1=\frac{\partial\bar z_1}{\partial \bar w_1}\,d\bar w_1+\frac{\partial\bar z_1}{\partial \bar w_2}\,d\bar w_2 \\
d\bar z_2=\frac{\partial\bar z_2}{\partial \bar w_1}\,d\bar w_1+\frac{\partial\bar z_2}{\partial \bar w_2}\,d\bar w_2
\end{array}
\right.,
\]
it follows that equation \eqref{dbaronh} is equivalent to
\begin{equation}
\label{dbarondd}
\begin{split}
\dbar_w u
&=\Big(\tilde \alpha_1\frac{\partial \bar z_1}{\partial \bar w_1}+\tilde \alpha_2\frac{\partial \bar z_2}{\partial \bar w_1}\Big)\,d\bar w_1+\Big(\tilde \alpha_1\frac{\partial \bar z_1}{\partial \bar w_2}+\tilde \alpha_2\frac{\partial \bar z_2}{\partial \bar w_2}\Big)\,d\bar w_2\\
&=\bar w_2\cdot\tilde \alpha_1\,d\bar w_1+(\bar w_1\cdot\tilde \alpha_1+\tilde \alpha_2)\,d\bar w_2\\
&=f_1\,d\bar w_1+f_2\,d\bar w_2\\
&=f
\end{split}
\end{equation}
on $\D\times\D^*$, where $u=v\circ\phi^{-1}$ and $\tilde \alpha_j=\alpha_j\circ\phi^{-1}$ for $j=1,2$. 

Note that
\[
\frac{\partial f_1}{\partial \bar w_2}=\frac{\partial }{\partial \bar w_2}(\bar w_2\cdot\tilde \alpha_1)=\tilde \alpha_1+\bar w_2\bar w_1\frac{\partial \alpha_1}{\partial \bar z_1}+\bar w_2\frac{\partial \alpha_1}{\partial \bar z_2}
\]
and
\[
\frac{\partial f_2}{\partial \bar w_1}=\frac{\partial }{\partial \bar w_1}(\bar w_1\cdot\tilde \alpha_1+\tilde \alpha_2)=\tilde \alpha_1+\bar w_1\bar w_2\frac{\partial \alpha_1}{\partial \bar z_1}+\bar w_2\frac{\partial \alpha_2}{\partial \bar z_1}.
\]
Since $\alpha$ is $\dbar$-closed on $\h$, it follows that $f$ is $\dbar$-closed on $\D\times\D^*$.

\subsection{An $L^p$ assumption on $\h$}

Based on the transformation \eqref{dbarondd}, a vanishing condition on $\alpha_j$ at the origin, 
\[
\|\alpha_j\|^p_{L^p_{-2}(\h)}=\int_{\h}|\alpha_j|^p|z_2|^{-2}\,dV(z)<\infty
\]
for $j=1,2$, implies that $f_1,f_2\in L^p(\D\times\D^*)$. This follows since
\[
\int_{\D\times\D^*}|\tilde \alpha_1|^p\cdot|w_2|^p\,dV(w) \le \int_{\D\times\D^*}|\tilde \alpha_1|^p\,dV(w)=\|\alpha_1\|^p_{L^p_{-2}(\h)}
\]
and
\[
\int_{\D\times\D^*}|\bar w_1\cdot\tilde \alpha_1+\tilde \alpha_2|^p\,dV(w) \le C_p\int_{\D\times\D^*}|\tilde \alpha_1|^p + |\tilde \alpha_2|^p\,dV(w)=C_p\left(\|\alpha_1\|^p_{L^p_{-2}(\h)}+\|\alpha_2\|^p_{L^p_{-2}(\h)}\right).
\]
In addition, a vanishing condition on derivatives of $\alpha_1$ at the origin, 

\[
\left\|\frac{\partial \alpha_1}{\partial \bar z_j}\right\|^p_{L^p_{-1}(\h)}=\int_{\h}\left|\frac{\partial \alpha_1}{\partial \bar z_j}\right|^p\cdot |z_2|^{-1}\,dV(z)<\infty
\]
for $j=1,2$, implies that $\sd f \in L^p(\D\times\D^*)$. This follows since for $1\le p<\infty$
\begin{align*}
\int_{\D\times\D^*}\left| \tilde \alpha_1+\bar w_1\bar w_2\frac{\partial \alpha_1}{\partial \bar z_1}+\bar w_2\frac{\partial \alpha_1}{\partial \bar z_2} \right|^p
&\le C_p\int_{\D\times\D^*}|\tilde \alpha_1|^p+|w_2|\cdot\left|\frac{\partial \alpha_1}{\partial \bar z_1}\right|^p+|w_2|\cdot \left|\frac{\partial \alpha_1}{\partial \bar z_2}\right|^p\\
\le C_p&\left(\|\alpha_1\|^p_{L^p_{-2}(\h)}+\left\|\frac{\partial \alpha_1}{\partial \bar z_1}\right\|^p_{L^p_{-1}(\h)}+\left\|\frac{\partial \alpha_1}{\partial \bar z_2}\right\|^p_{L^p_{-1}(\h)}\right).
\end{align*}

Thus, the following $L^p$ estimate for a solution of $\dbar$ on $\h$ holds.

\begin{theorem}
\label{HarLp'}
Let $v$, $\alpha$ be as in \eqref{dbaronh} and $u$, $f$ be as in \eqref{dbarondd}. Suppose $\alpha$ is $\dbar$-closed in the weak sense on $\h$. For $1\le p<\infty$, assume that
\begin{enumerate}
\item $\alpha_1,\alpha_2\in L^p_{-2}(\h)$,
\item $\partial \alpha_2/\partial \bar z_1=\partial \alpha_1/\partial \bar z_2\in L^p_{-1}(\h)$ and $\partial \alpha_1/\partial \bar z_1\in L^p_{-1}(\h)$.
\end{enumerate}
Then there is a weak solution $v=u\circ\phi=T(f)\circ\phi$, where $T$ is the solution operator in \eqref{T1forDxD*}, satisfying the $L^p$ estimate
\[
\|v\|_{L^p(\h)}\le C\left(\sum_{j=1}^2\|\alpha_j\|_{L^p_{-2}(\h)}+\sum_{j=1}^2\|\partial \alpha_1/\partial \bar z_j\|_{L^p_{-1}(\h)}\right)
\]
for some constant $C>0$ independent of $\alpha$.
\end{theorem}

\subsection{An example}\label{SS:example} The $L^p$ estimates of the solution $v$ given in Theorem \ref{HarLp'}, and the cannonical solution $v_{\text{can}}$ on $\h$ can be compared via the simple example
\[
\alpha=d\bar z_2.
\]
Verifying that $\alpha$ satisfies the conditions in Theorem \ref{HarLp'} is easy. Thus the solution $v$ in Theorem \ref{HarLp'} belongs to $L^p(\h)$ for $1\le p<\infty$.

On the other hand, we claim the $L^2$-minimal solution of $\dbar v= \alpha$ is
\[
v_{\text{can}}=\bar z_2-cz_2^{-1}
\]
for some nonzero constant $c$. Clearly $\dbar v_{\text{can}}=\alpha$. To see that $v_{\text{can}}$ is orthogonal to holomorphic functions on $\h$, it suffices to take its inner product with the orthogonal basis $\{z_1^nz_2^m\}$ on $\h$, for so-called allowable indices $(n,m)\in \Z^+\times\Z$. See Sections   of \cite{EdhMcN17} for the definition of allowable indices and Section 5 of that paper of that paper for a proof that $\left\langle v_{\text{can}}, z^\alpha\right\rangle =0$ for all allowable exponents $\alpha$.
Note that $\alpha\in L^p(\h)$. On the other hand, the proof of Proposition 5.5 in \cite{EdhMcN17} implies that  $v_{\text{can}}\notin L^p(\h)$ for $p\ge4$. Thus, $v$  behaves better than $v_{\text{can}}$ in terms of $L^p$ regularity. 

At the operator level, it follows that the canonical solution operator for $\dbar$ on $\h$ doesn't map $L^p$ $\dbar$-closed $(0,1)$-form to $L^p$ functions for $p\ge4$. This is consistent with results on the Bergman projection on $\h$, see \cite{EdhMcN16}, \cite{ChaZey16}, and \cite{Chen17}.

\subsection{Extra condition}\label{SS:extra}

An extra condition on $\alpha$, namely
\begin{equation}
\label{f2=0}
\bar z_1\cdot\alpha_1+\bar z_2\cdot\alpha_2=0,
\end{equation}
and \eqref{dbarondd} shows that $f=f_1d\bar w_1$ on $\D\times\D^*$. By \eqref{T1forDxD*}, $T(f)$ only involves the second term. Thus a better $L^p$ estimate holds in this case:

\begin{theorem}
Let $v$, $\alpha$ be as in \eqref{dbaronh} and $u$, $f$ be as in \eqref{dbarondd}. Suppose $\alpha$ is $\dbar$-closed in the weak sense on $\h$ and satisfies \eqref{f2=0}. For $1\le p<\infty$, assume that $\alpha_1\in L^p(\h)$. Then there is a weak solution $v=u\circ\phi=T(f)\circ\phi$, where $T$ is the solution operator in \eqref{T1forDxD*}, satisfying the $L^p$ estimate
\[
\|v\|_{L^p(\h)}\le C\|\alpha_1\|_{L^p(\h)}
\]
for a constant $C>0$ independent of $\alpha$.
\end{theorem}

\begin{proof}
Starting with \eqref{T1forDxD*} and applying Lemma \ref{lem135}, it follows  that
\begin{equation}
\label{lpboundednessondd}
\int_{\D\times\D^*}|T(f)|^p\,dV(w) \le C\int_{\D\times\D^*}|f_1|^p\,dV(w).
\end{equation}

(i) When $p \ge 2$, it holds that
\[
\int_{\D\times\D^*}|\tilde \alpha_1|^p\cdot|w_2|^p\,dV(w) \le \int_{\D\times\D^*}|\tilde \alpha_1|^p\cdot|w_2|^2\,dV(w)=\int_{\h}|\alpha_1|^p\,dV(z).
\]
Since $f_1=\bar w_2\cdot\tilde \alpha_1$ and $v=u\circ\phi=T(f)\circ\phi$, by \eqref{lpboundednessondd}, it follows
\[
\int_{\h}|v|^p\,dV(z)=\int_{\D\times\D^*}|u|^p\cdot |w_2|^2\,dV(w)\le \int_{\D\times\D^*}|T(f)|^p\,dV(w) \le C\int_{\h}|\alpha_1|^p\,dV(z).
\]

(ii) When $1 \le p<2$, consider
$\dbar_w (u \cdot w_2)=w_2\cdot\dbar_w u=w_2\cdot f=w_2\cdot f_1\,d\bar w_1$.
By \eqref{T1forDxD*},  $T(w_2\cdot f)=w_2\cdot T(f)=w_2\cdot u$. Therefore, replacing $f$ by $w_2\cdot f$ in \eqref{lpboundednessondd}, 
\[
\int_{\D\times\D^*}|u\cdot w_2|^p\,dV(w) \le C\int_{\D\times\D^*}|f_1\cdot w_2|^p\,dV(w)=C\int_{\D\times\D^*}|\tilde \alpha_1|^p\cdot|w_2|^{2p}\,dV(w).
\]
Since $1 \le p<2$, 
\[
\int_{\h}|v|^p\,dV(z)=\int_{\D\times\D^*}|u|^p\cdot |w_2|^2\,dV(w) \le \int_{\D\times\D^*}|u|^p\cdot|w_2|^p\,dV(w)
\]
and
\[
\int_{\D\times\D^*}|\tilde \alpha_1|^p\cdot|w_2|^{2p}\,dV(w) \le \int_{\D\times\D^*}|\tilde \alpha_1|^p\cdot|w_2|^2\,dV(w)=\int_{\h}|\alpha_1|^p\,dV(z).
\]
Hence, 
\[
\int_{\h}|v|^p\,dV(z) \le C\int_{\h}|\alpha_1|^p\,dV(z).
\]
\end{proof}


\bibliographystyle{alpha}
\bibliography{ChenMcNeal16}

\begin{thebibliography}{HKR14}

\bibitem[Cha91]{ChaumatChollet91}
Anne-Marie Chaumat, Jacques \&~Chollett.
\newblock R\' egularit\' e h\"old\' erienne de l'op\' erateur $\dbar$ sur le
  triangle de {H}artogs.
\newblock {\em Ann. Inst. Fourier (Grenoble)}, 41(4):867--882, 1991.

\bibitem[Cha93]{ChaumatChollet93}
Anne-Marie Chaumat, Jacques \&~Chollett.
\newblock Noyaux pour r\' esoudre l'equation $\dbar$ dans des classes
  ultradiff\' erntiables sure des compacts irr\' eguliers de $\mathbb{C}^n$.
\newblock {\em Princeton University Press, Math. Notes 38}, pages 205--226,
  1993.

\bibitem[Che17]{Chen17}
Liwei Chen.
\newblock The ${L}^p$ boundedness of the {B}ergman projection for a class of
  bounded {H}artogs domains.
\newblock {\em J. Math. Anal. Appl.}, 448:598--610, 2017.

\bibitem[Chr96]{Christ96}
Michael Christ.
\newblock Global ${C}^\infty$ irregularity of the $\dbar$-{N}eumann problem for
  worm domains.
\newblock {\em J. Amer. Math. Soc.}, 9(4):1171--1185, 1996.

\bibitem[CNS92]{ChaNagSte}
D.C. Chang, A.~Nagel, and E.M. Stein.
\newblock Estimates for the $\dbar$ {N}eumann problem for pseudoconvex domains
  of finite type in ${\C}^2$.
\newblock {\em Acta Math.}, 169:153--228, 1992.

\bibitem[CZ16]{ChaZey16}
Debraj Chakrabarti and Y.~Zeytuncu.
\newblock ${L}^p$ mapping properties of the {B}ergman projection on the
  {H}artogs triangle.
\newblock {\em Proc. Amer. Math. Soc.}, 144(4):1643--1653, 2016.

\bibitem[EM16]{EdhMcN16}
Luke~D. Edholm and Jeffery~D. McNeal.
\newblock The {B}ergman projection on fat {H}artogs triangles: ${L}^p$
  boundedness.
\newblock {\em Proc. Amer. Math. Soc.}, 144(5):2185--2196, 2016.

\bibitem[EM17]{EdhMcN17}
Luke~D. Edholm and Jeffery~D. McNeal.
\newblock Bergman subspaces and subkernels: degenerate ${L}^p$ mapping and
  zeroes.
\newblock {\em J. Geom. Anal.}, 27(4):2658--2683, 2017.

\bibitem[Eva98]{Evans98}
Lawrence~C. Evans.
\newblock {\em Partial differential equations}, volume~19 of {\em Graduate
  Studies in Mathematics}.
\newblock American Mathematical Society, 1998.

\bibitem[FK88]{FefKohn88}
Charles Fefferman and Joseph~J. Kohn.
\newblock H\" older estimates on domains of complex dimension two and on three
  dimensional {CR} manifolds.
\newblock {\em Adv. Math.}, 69:233--303, 1988.

\bibitem[FLZ11]{FornaessLeeZhang}
John~Erik Forn{\ae}ss, Lina Lee, and Yuan Zhang.
\newblock On supnorm estimates for $\bar \partial$ on infinite type convex
  domains in $\mathbb{C}^2$.
\newblock {\em Journal of Geometric Analysis}, 21(3):495--512, 2011.

\bibitem[GL70]{GrauertLieb}
Hans Grauert and Ingo Lieb.
\newblock Das {R}amirezsche {I}ntegral und die {L}\" osung der {G}leichung
  $\dbar f=\alpha$ im {B}ereich der {F}ormen.
\newblock {\em Rice Univ. Studies}, 56:29--50, 1970.

\bibitem[Hen69]{Henkin69}
G.~M. Henkin.
\newblock Integral representations of functions holomorphic in strictly
  pseudoconvex domains and some applications.
\newblock {\em Mat. Sb.}, 78:611--632, 1969.

\bibitem[Hen71]{Henkin71}
G.~M. Henkin.
\newblock A uniform estimate for the solution of the $\bar \partial$-problem in
  a {W}eil region.
\newblock {\em (Russian) Uspehi Mat. Nauk}, 26(3(159)):211--212, 1971.

\bibitem[HKR14]{HaKhaRai}
Ly~Kim Ha, Tran~Vu Khanh, and Andrew Raich.
\newblock ${L}^p$ estimates for the $\dbar$-equation on a class of infinite
  type domains.
\newblock {\em Inter. J. Math.}, 25(11):15 pp., 2014.

\bibitem[HL84]{HenkinLeiterer}
Gennadi Henkin and Jurgen Leiterer.
\newblock {\em Theory of functions on complex manifolds}.
\newblock Birkh\"auser Boston, 1984.

\bibitem[Ker71]{Kerzman71}
Norberto Kerzman.
\newblock H\" older and ${L}^p$ estimates for solutions of $\dbar u=f$ in
  strongly pseudconvex domains.
\newblock {\em Comm. Pure Appl. Math.}, 24:301--379, 1971.

\bibitem[Koh73]{Kohn73}
Joseph~J. Kohn.
\newblock Global regularity for $\dbar$ on weakly pseudo-convex manifolds.
\newblock {\em Trans. Amer. Math. Soc.}, 181:273--292, 1973.

\bibitem[Lie70]{Lieb70}
Ingo Lieb.
\newblock Die {C}auchy-{R}iemann {D}ifferentialgleichenungen auf streng
  pseudo-konvexen {G}ebieten.
\newblock {\em Math. Ann.}, 190:6--44, 1970.

\bibitem[LS13]{LanSte13}
Loredana Lanzani and Eli Stein.
\newblock Cauchy-type integral in several complex variables.
\newblock {\em Bull. Math. Sci.}, 3(241--285), 2013.

\bibitem[Ma92]{MaMichel}
J.~Ma, L. \&~Michel.
\newblock ${C}^{k+\alpha}$ estimates for the $\dbar$-equation on the {H}artogs
  triangle.
\newblock {\em Math. Ann.}, 294(4):661--675, 1992.

\bibitem[Ovr71]{Ovrelid}
Nils Ovrelid.
\newblock Integral representation formulas and ${L}^p$ estimates for the
  $\dbar$-equation.
\newblock {\em Math. Scand.}, 29:137--160, 1971.

\bibitem[Ran86]{Range86}
R.~Michael Range.
\newblock {\em Holomorphic functions and integral representations in several
  complex variables}, volume 108 of {\em Graduate Texts in Mathematics}.
\newblock Springer-Verlag, New York, 1986.

\bibitem[Ran13]{Range13}
R.~Michael Range.
\newblock An integral kernel for weakly pseudoconvex domains.
\newblock {\em Math. Ann.}, 356(2):793--808, 2013.

\bibitem[RS73]{RangeSiu}
Michael~R. Range and Yum-Tong Siu.
\newblock Uniform estimates for the $\bar \partial$-equation on domains with
  piecewise smooth strictly pseudoconvex boundaries.
\newblock {\em Math. Ann.}, 206:325--354, 1973.

\end{thebibliography}

\end{document}